\documentclass[10 pt, a4 paper, oneside, openany]{article}
\usepackage[T1]{fontenc}
\usepackage{amsmath}
\usepackage[utf8]{inputenc}
\usepackage[english]{babel}
\usepackage{graphicx}
\usepackage{enumerate}
\usepackage{enumitem}
\usepackage{amsthm}
\usepackage{amsfonts}
\usepackage{amssymb}
\usepackage{amsthm}
\usepackage{wrapfig}
\usepackage{pgfplots}
\pgfplotsset{compat=1.15}
\usepackage{mathrsfs}
\usepackage{hyperref}
\usetikzlibrary{arrows}
\usepackage[swapnames, nouppercase]{frontespizio} 
\usepackage{geometry}
\usepackage{tikz}
\usepackage{float}
\usepackage{appendix}
\usepackage{setspace}
\usepackage{xcolor}
\allowdisplaybreaks

\newgeometry{right=1.5cm} 

\newtheorem{teorema}{Theorem}[section]

\newtheorem{lemma}[teorema]{Lemma}
\newtheorem{proposizione}[teorema]{Proposition}

\theoremstyle{definition}
\newtheorem{osservazione}[teorema]{Remark}
\newtheorem{definition}[teorema]{Definition}

\begin{document}
\title{On the structure of stable constant anisotropic mean curvature hypersurfaces}
\date{}
\author{Claudia Pontuale}

\maketitle

\begin{abstract}
We prove that a stable noncompact hypersurface with constant anisotropic mean curvature in $\mathbb{R}^{n+1}$ has only one end for $n \le 6$, under a natural ellipticity condition on the anisotropy. This provides an anisotropic counterpart of the one-end theorem for stable constant mean curvature hypersurfaces. We remark that the result is also new in $\mathbb R^7$ (n=6),  in the case of anisotropic minimal hypersurfaces. The proof relies on the existence of a Sobolev-type inequality.
\end{abstract}

\section*{Introduction}
In the study of complete noncompact hypersurfaces, the number of ends provides a natural way to describe their behavior at infinity.

 In their seminal work \cite{caoshenzhu}, Cao, Shen and Zhu showed that a complete noncompact stable minimal hypersurface immersed in Euclidean space must have only one end. Their argument relies on a theorem of Schoen and Yau \cite{schoenyau}, which states that such hypersurfaces do not admit nonconstant harmonic functions with finite Dirichlet energy. Assuming the existence of more than one end, they constructed such a function, thus obtaining a contradiction.

In the following years, several extensions of this result were obtained. \\In the case of finite-index minimal hypersurfaces, Li and Wang \cite{liwang}
proved that such hypersurfaces have only finitely many ends. 
 \\This relation between index and topology was further
developed by Chao Li \cite{chaoli}, who obtained quantitative lower bounds for
the Morse index in terms of the number of ends and the first Betti number.

The constant mean curvature case requires a separate treatment. Indeed, in this setting,
the presence of the mean curvature terms leads to
dimension-dependent restrictions.\\ In this direction, Cheng, Cheung and Zhou (\cite{chengcheungzhou}) proved that a complete, noncompact, weakly stable CMC hypersurface has only one end in $\mathbb R^{n+1}, n \le 5$. \\This result was later generalized by Fu and Li \cite{fuli} to dimension 6, considering only the case of strongly stable hypersurfaces. \\More recently, Chen, Hong and Li \cite{chenhongli} showed that a finite-index CMC hypersurface has only finitely many ends up to dimension 6.

A key idea behind these contributions is the use of harmonic functions to detect the
topology at infinity. This point of view is closely related to the work of
Li and Tam \cite{litam,litam2}.

Anisotropic generalizations of these results have also attracted considerable attention. 
In this framework, one fixes a smooth, positive and strictly convex function 
$\varphi :  \mathbb R^{n+1} \setminus \{0\}\to \mathbb R$ and considers the anisotropic area functional of a hypersurface $M$
\begin{equation*}
\label{eq:area-intro}
\mathcal A_\varphi(M) = \int_M \varphi(\nu)\, d\mu,
\end{equation*}
where $\nu$ denotes the unit normal to $M$ and $d \mu$ the volume element. 

In complete analogy with the isotropic case, critical points of \(\mathcal A_\varphi\)
under volume-preserving variations are hypersurfaces whose anisotropic mean
curvature is constant (see Section 1 for a definition). We denote this constant by \(H_\varphi\), and refer to the
corresponding hypersurfaces as CAMC hypersurfaces.

The special case \(H_\varphi=0\) corresponds to anisotropic minimal hypersurfaces.
Throughout the paper, unless otherwise specified, CAMC hypersurfaces include this
minimal case. The corresponding notions of stability, with and without the volume
constraint, are recalled in Section 1.

Such hypersurfaces naturally arise in models of crystalline interfaces (see the Introduction in \cite{koisopalmer} for a brief explanation of the phenomena), where the surface energy depends on the orientation of the normal and 
have been extensively studied in, for example, \cite{galvezmiratassi}, \cite{heli}, \cite{palmer}, \cite{pham}.

In $\mathbb R^4$, Chodosh and Li \cite{chodoshli} established a one-end theorem for anisotropic minimal hypersurfaces, which was subsequently generalized by Li and Xia \cite{lixia} to anisotropic minimal hypersurfaces in $\mathbb R^5, \mathbb R^6$. 
Both these results require on a \emph{pinching hypothesis} on the Hessian of the anisotropy $\varphi$, namely
\begin{equation}
\label{eq:pinching-intro-lambda}
  \lambda_n^{min} |v|^2 \le D^2 \varphi (\nu)(v,v) \le \lambda_n^ {max}|v|^2, \forall v \in \nu ^ \perp
\end{equation}
with $ \lambda_n^{min}$ and $\lambda_n^ {max}$ the lowest and greatest eigenvalues of $D^2 \varphi(\nu)$.

This condition implies the uniform ellipticity of the anisotropic area functional and ensures that the induced anisotropic geometry is comparable to the Euclidean one.

Motivated by these developments, we prove the following result.

\begin{teorema}
    \label{th:main-intro}
    Let $M$ be a weakly stable constant anisotropic mean curvature hypersurface in $\mathbb R^{n+1}$ ($n \le 6$) and assume that there exists a constant $\Lambda_n < \frac{n^2}{n^2-1}$ such that  \begin{equation*} 
    |v|^2 \le D^2 \varphi (\nu)(v,v) \le  \Lambda_n|v|^2, \forall v \in \nu ^\perp
\end{equation*} 
Then, $M$ has only one end.
\end{teorema}
This theorem shows that, in low dimensions, stability alone controls the topology of constant anisotropic mean curvature hypersurfaces, forcing them to have exactly one end.

It is worth noting that in the isotropic CMC case, our argument gives a different proof of the 
one-end theorem of Fu and Li \cite{fuli} for stable CMC hypersurfaces. Moreover, in the anisotropic minimal case 
\((H_\varphi=0)\), it extends the one-end results of Chodosh--Li 
\cite{chodoshli} and Li--Xia \cite{lixia} up to dimension 
\(n=6\). Finally, for nonzero constant 
anisotropic mean curvature hypersurfaces, \(H_\varphi\neq0\), the result 
appears to be new in every dimension.

\section{Preliminaries on anisotropic functionals}

This section is devoted to the proof of some preliminary results for anisotropic functionals. 

Throughout this article, $\varphi: \mathbb R^{n+1} \setminus \{0\} \to (0, \infty)$ is a 1-homogeneous $C^3_{loc}$ function (i.e. $\varphi (tv)= t \varphi(v)$, for each $t >0$ and for each $v \in \mathbb  R^{n+1} \setminus \{0\}$).\\ Given an immersion $X: M \hookrightarrow \mathbb R^{n+1}$ we define the \emph{anisotropic area functional} as
\begin{equation*}
    \label{eq:anis-areafun}
    \mathcal A_\varphi(M)= \int_{M} \varphi( \nu(x)) \,d\mu,
\end{equation*}
where $\nu(x)$ is the unit normal at the point $x \in M$ and $d\mu$ is the volume element. We say that the functional $\mathcal A_\varphi$ is \emph{elliptic} if $D^2 \varphi (\nu)$ is positive definite for all $\nu \in \mathbb S^n$.

We begin by deriving the first and second variation formulae for hypersurfaces with constant anisotropic mean curvature. The first variation formula is standard, and the second variation formula is well known in the anisotropic minimal case (see, e.g., \cite{chodoshli}). However, to the best of our knowledge, a complete computation for CAMC hypersurfaces under the volume constraint is not readily available in the literature. For this reason, we include the details here, adapting the argument of \cite{chodoshli} for anisotropic minimal hypersurfaces.

We consider a compactly-supported, volume-preserving, normal variation of $X$,
\begin{equation}
    \label{eq:xt}
    X_t= X + t f \nu,
\end{equation}
with $f \in C_c^1(M \setminus \partial M)$ and $\int_Mf=0$. 

In the following computations, $\bar \nabla $ will be the connection of $\mathbb R^{n+1}$ and $\nabla$ the induced connection on the hypersurface $M$. We use the notation $D$ (resp. $D^2$) to represent the gradient (resp. the Hessian) of a function defined in $\mathbb R^{n+1}$. Finally, we denote with $\xi$ the variational vector field, that is,
\begin{equation*}
    \label{eq:xi}
    \xi=  \left. \frac{d}{dt}  \right|_{t=0} X_t.
\end{equation*}
We have
\begin{align*}
\begin{split}
    \label{eq:variation1}
    \left. \frac{d}{dt}  \right|_{t=0} \mathcal A _\varphi(M_t)&= \int_{M} \varphi (\nu) H f \,d\mu + \int_{M} \left( \left. \frac{d}{dt}  \right|_{t=0} \varphi (\nu) \right )\, d\mu=\\
    &= \int_{M} (\varphi ( \nu) H f - \langle D \varphi (\nu), \nabla f \rangle )\,d\mu
\end{split}\end{align*}
where, in the last line, we used that, for a variation of the form \eqref{eq:xt}, 
\begin{equation*}
  \left. \frac{d}{dt}  \right|_{t=0} \nu= -\nabla f.  
\end{equation*} 
Now, applying the divergence theorem, and using that $f$ has compact support, one gets
\begin{equation*}\begin{split}
    \label{eq:variation2}
    \left. \frac{d}{dt}  \right|_{t=0} \mathcal A_ \varphi (M_t)&= \int_{M}\left( H \varphi(\nu) + \operatorname{div}_M (D \varphi (\nu)^T)\right)f \, d \mu=\\
    &=\int_{M} \left( H \varphi (\nu) + \operatorname{div}_M (D \varphi (\nu)- \langle D \varphi(\nu), \nu \rangle \nu \right))f \, d    \mu=\\
    &= \int_{M} \left(H \varphi (\nu) + \operatorname{div}_M (D \varphi (\nu) \right)- \langle D \varphi (\nu), \nu \rangle Hf \, d\mu,
\end{split}\end{equation*}
with $^T$ denoting the tangential to $M$.\\Applying Euler's theorem for homogeneous functions to $\varphi (\nu)$, we obtain $\langle D \varphi (\nu), \nu \rangle = \varphi(\nu)$.
\\Thus,
\begin{equation*}
    \label{eq:variation3}
    \left. \frac{d}{dt}  \right|_{t=0} \mathcal A_\varphi (M_t)= \int_{M} \operatorname{div}_M (D \varphi (\nu)) f \, d\mu.
\end{equation*}
In analogy to the first variation formula for the classical area functional, one defines the \emph{anisotropic mean curvature} as
\begin{equation}
    \label{eq:meancurv-anis}
    H_\varphi= \operatorname{div}_M (D \varphi (\nu)).  
\end{equation}
We shall now discuss the second variation. We let $\{e_1,e_2,...,e_n\}$ denote an orthonormal frame on $M$ and note that
\begin{equation*}
    \begin{split}
        \operatorname{div}_M( D \varphi (\nu))&= \sum_{i=1}^n  \langle \nabla_{e_i} (D \varphi (\nu)) , e_i \rangle =\\
        &= \sum_{i=1}^n D^2 \varphi (\nu) \left(A(e_i), e_i\right),
    \end{split}
\end{equation*}
where $A$ is the second fundamental form of $M$.
We define $\psi (\nu): T \mathbb R^{n+1} \to T \mathbb R^{n+1}$ by $\psi (\nu): X \to D^2 \varphi (\nu)[X, \cdot]$.
This implies
\begin{equation}
    \label{eq:meancurv-anis2}
    H_\varphi= \operatorname{tr}_M (\psi (\nu) A).
\end{equation}
Differentiating \eqref{eq:meancurv-anis2},
\begin{equation}
    \begin{split}
        \label{eq:meancurv-anis3}
        \left. \frac{d}{dt}  \right|_{t=0} H_\varphi &= \operatorname{tr}_M \left(  \left. \frac{d}{dt}  \right|_{t=0} \psi(\nu) A\right)  + \operatorname{tr}_M \left ( \psi(\nu) \left. \frac{d}{dt}  \right|_{t=0} A \right ).
    \end{split}
\end{equation}
Note that the first term in \eqref{eq:meancurv-anis3} is
\begin{equation}
    \label{eq:meancurv-anis4}
    -\operatorname{tr}_M \left( \langle D \psi (\nu), \nabla f \rangle A\right).
\end{equation}
To compute the second term,
\begin{equation}\begin{split}
    \label{eq:tube1}
    \left. \frac{d}{dt}  \right|_{t=0}  \langle \bar \nabla_{e_i} e_j, \nu \rangle = \langle \bar \nabla_{\xi} \bar \nabla _{e_i}e_j, \nu \rangle - \langle \bar \nabla_{e_i} e_j , \nabla f \rangle.
\end{split}    
\end{equation}
Note that, as the ambient space is flat,
\begin{equation*}\begin{split}
 \label{eq:tube2}   
  \bar \nabla_{\xi} \bar \nabla _{e_i}e_j &= \bar \nabla_ {e_i} \bar \nabla_{\xi} e_j= \bar \nabla_{e_i} \bar \nabla_ {e_j} \xi=\\
 &= \bar \nabla_{e_i} \bar \nabla_{e_j}(f \nu)= \bar \nabla_{e_i} (e_j(f) \nu + f \bar \nabla_{e_j} \nu)= \\
 &=e_i (e_j(f)) \nu + e_j(f) \bar \nabla_{e_i} \nu +e_i(f) A(e_j)+ f \bar \nabla _{e_i}(A(e_j)).
 \end{split}\end{equation*}
Scalar multiplication with $\nu$ gives
\begin{equation}
    \begin{split}
        \label{eq:tube3}
        \langle  \bar \nabla_{\xi} \bar \nabla _{e_i}e_j, \nu \rangle = e_i (e_j(f))+ f\langle  A^2(e_j), e_i \rangle.
    \end{split}
\end{equation}
Combining \eqref{eq:tube1} and \eqref{eq:tube3} and using the definition of Hessian, 
\begin{equation}
    \label{eq:tubeformula}
     \left. \frac{d}{dt}  \right|_{t=0} A= - A^2f -\nabla ^2 f.
\end{equation}
Thus, using \eqref{eq:meancurv-anis4} and \eqref{eq:tubeformula} in \eqref{eq:meancurv-anis3}, one gets
\begin{equation} \begin{split}
    \label{eq:2variation}
     \left. \frac{d}{dt}\right|_{t=0} H_\varphi =\operatorname{tr}_M(- \psi (\nu) \nabla^2 f - \psi (\nu) A^2 f - \langle D \psi (\nu), \nabla f \rangle  A).
\end{split} \end{equation}
Integrating over $M$,
\begin{equation*}
    \int_M \operatorname{tr}_M(- \psi(\nu) \nabla^2 f)= \int_M \langle \nabla f, \psi(\nu) \nabla f \rangle + f \langle D \psi (\nu), \nabla f \rangle A.
\end{equation*}
Finally, as the variation is volume-preserving (and this, in particular, implies $\int_M f \, d \mu=0$),
\begin{equation}
    \begin{split}
        \label{eq:2variation5}
     \left. \frac{d^2}{dt^2}  \right|_{t=0} \mathcal A_ \varphi (M_t) &= \int_M \langle \nabla f , \psi(\nu) \nabla f \rangle - \operatorname{tr}_M (\psi(\nu)A^2)f^2.
     \end{split}\end{equation}

In complete analogy with the isotropic case, we distinguish between two notions of stability.
A CAMC hypersurface is said to be \emph{strongly stable} if its second variation \eqref{eq:2variation5} is nonnegative
for every $f \in C_c^1(M)$. It is said to be \emph{weakly stable} if the same condition holds
for every $f \in C_c^1(M)$ satisfying $\int_M f\,d\mu=0$.

For anisotropic minimal hypersurfaces, the natural notion is stability, since no volume constraint is imposed.
For nonzero constant anisotropic mean curvature hypersurfaces, the natural notion is weak stability,
corresponding to volume-preserving variations.(see e.g. \cite{palmer} for an anisotropic analogue of \cite{barbosadocarmo}).

Let $\lambda_n^{\min}$ and $\lambda_n^{\max}$ denote respectively the smallest and largest
eigenvalues of $D^2\varphi(\nu)$. Under the pinching condition
\eqref{eq:pinching-intro-lambda}, the stability inequality above implies
\begin{equation}
    \label{stability}
    \int_M |\nabla f|^2 \ge \int_M \frac{1}{\Lambda_n} |A|^2 f^2,
\end{equation}
for the same class of test functions, where
\[
\Lambda_n=\frac{\lambda_n^{max}}{\lambda_n^{min}}.
\]

\begin{definition}
A CAMC hypersurface with pinching condition \eqref{eq:pinching-intro-lambda} is said to be \emph{stable} if \eqref{stability} is satisfied for each $f \in C_c^1(M)$, \emph{weakly stable} if \eqref{stability} is satisfied for each $f \in C_c^1(M)$ with $\int_M f \,d\mu=0$.
\end{definition}

Throughout the paper, all hypersurfaces are assumed to satisfy the pinching condition \eqref{eq:pinching-intro-lambda}.

Moreover we remark that, in what follows, we frequently use the fact that if a CAMC hypersurface $M$ is weakly stable, then there exists a compact set 
$K \subset M$ such that $M$ is strongly stable outside $K$. Proof of this statement for isotropic minimal hypersurfaces can be found in \cite[Proposition 3]{cmcbellettini}, and the same argument applies unchanged in the present paper.

The stability inequality \eqref{stability} allows us to obtain the following Sobolev type inequality, which is a fundamental tool in our setting. 
\begin{teorema}
\label{th:sobolev}
Assume that $M$ is a complete $n$-dimensional weakly stable hypersurface with constant anisotropic mean curvature $H_\varphi \neq 0$ immersed in $\mathbb R^{n+1}$. Suppose that $M$ satisfies the uniform ellipticity assumption \eqref{eq:pinching-intro-lambda}.\\
Then, there exists a positive constant $C(n, \varphi)$ and a compact subset $K \subset M$ such that for each $f \in C_c^1(M \setminus K)$, the following Sobolev inequality holds.
\begin{equation}
\label{eq:sobolev}
\left( \int_M |f|^{\frac{2n}{n-1}} \right)^{\frac{n-1}{n}}
\le C(n, \varphi) \int_M |\nabla f|^2.
\end{equation}   
If $M$ is strongly stable, $K = \emptyset$.
\end{teorema}

\begin{proof}
By the definition of the anisotropic mean curvature and the Cauchy--Schwarz
inequality for tensors, we have
\begin{equation*}
\label{eq:sobolev-estimate1}
    |H_\varphi|
    =
    |\operatorname{tr}(\psi(\nu)A)|
    \le
    |\psi(\nu)||A|
    \le
    \sqrt n\,\lambda_n^{\max}|A|.
\end{equation*}
Then,
\begin{equation*}
    \label{eq:sobolev-estimate2}
    |A|^2 \ge \frac{H_\varphi^2}{n (\lambda_n^{max})^2}.
\end{equation*}
Choosing 
\begin{equation*}
    \label{eq:sobolev-estimate3}
    C(n, \varphi)= \frac{H_\varphi^2}{n (\lambda_n^{max})^3} \lambda_n^{min},
\end{equation*}
one gets from the stability inequality \eqref{stability},
\begin{equation}
    \label{eq:sobolev-estimate4}
    \int_M |\nabla f|^2 \ge C(n, \varphi) \int_M f^2,
\end{equation}
for all $f \in C^\infty _c (M \setminus K)$.\\
 From the Michael-Simon Sobolev inequality (see \cite{michealsimon}), we get that there exists a constant $C(n)$ such that for any 
\( f \in C^1_c(M) \),
\begin{equation}
\label{eq:sobolev1}
C(n) \left( \int_M |f|^{\frac{n}{n-1}} \right)^{\frac{n-1}{n}}
\le \int_M \big( |\nabla f| + |fH| \big),
\end{equation}
where $H$ is the classical mean curvature.
Replacing \( f \) by \( f^{2} \), we find
\begin{equation}\label{eq:sobolev2}
C(n) \left( \int_M |f|^{\frac{2n}{n-1}} \right)^{\frac{n-1}{n}}
\le \int_M 
 2|f| |\nabla f|
+ \int_M f^{2} |H|.
\end{equation}
First note that,
\begin{equation*}
    \label{eq:sobolev1.1}
    \int_M |f||\nabla f| \le \left(\int_M f^2\right)^\frac{1}{2} \left(\int_M |\nabla f|^2\right)^\frac{1}{2} \le C(n, \varphi) \int_M |\nabla f|^2, \forall f \in C^1_c(M \setminus K)
\end{equation*}
where the last inequality comes from \eqref{eq:sobolev-estimate4}.\\
 The second integral on the right hand side of \eqref{eq:sobolev2} is
\begin{equation}
    \label{eq:sobolev3}
    \int_M |f|^{2} |H| \le \left( \int_M |f|^2 \right)^\frac{1}{2} \left(\int_Mf^2 |H|^2\right)^\frac{1}{2} \le C(n,\varphi) \int_M | \nabla f |^2, \forall f \in C_c^1(M \setminus K),
\end{equation}
where the last inequality is the definition of stability and inequality \eqref{eq:sobolev-estimate4} (with a little abuse of notation, we denote the constant again by $C(n, \varphi)$).\\
Finally, \eqref{eq:sobolev2} and \eqref{eq:sobolev3} give the result.
\end{proof}
\begin{osservazione}
We point out that the Sobolev inequality above relies on the assumption
\(H_\varphi\neq 0\). In the anisotropic minimal case \(H_\varphi=0\), a stronger
Sobolev inequality 
\begin{equation}
\label{eq:sobolev-minime}
\left(\int_M |f|^{\frac{2n}{n-2}}\right)^{\frac{n-2}{n}}
\le C(n,\varphi)\int_M |\nabla f|^2,
f\in C_c^1(M), n \ge 3. 
\end{equation}
is available assuming uniform ellipticity of the anisotropy, as proved by
Chodosh and Li \cite{chodoshli}. When \(H_\varphi\neq0\), (weak stability) this stronger inequality \eqref{eq:sobolev-minime}
can also be obtained in dimensions \(n\ge3\). However, the advantage of the
weaker inequality proved here is that it also covers the two-dimensional case.
\end{osservazione}

\section{One-endedness of CAMC hypersurfaces}
This section is devoted to the proof of our main theorem. The overall strategy follows the approach of Cao, Shen, and Zhu. 

First, using Theorem \ref{th:sobolev}, we prove that each end of the hypersurface has infinite volume (Theorem \ref{th:infinitevol}). Second, under the assumption that $M$ has at least two ends, we deduce the existence of a bounded harmonic function with finite energy (Proposition \ref{pr:caoshenzhu-lemma2}). 
Finally, a Bochner-type argument combined with stability rules out such functions, forcing the hypersurface to have only one end (Theorem \ref{th:fuli}).

We remark that in \cite{caoshenzhu}, stability was not needed to establish the infiniteness of the volume of the ends (Theorem \ref{th:infinitevol}) or the existence of a non-constant harmonic function (Proposition \ref{pr:caoshenzhu-lemma2}), since minimal hypersurfaces already satisfy a suitable Sobolev inequality. It was, however, required to rule out the existence of such harmonic functions (see Lemma 3 in \cite{caoshenzhu} or \cite{schoenyau}).

\begin{teorema}
\label{th:infinitevol}
Suppose \( M \to \mathbb{R}^{n+1} \) is a two-sided, complete, weakly stable immersion of a CAMC hypersurface, and let \( K \) be a compact subset of \( M \). 
Then each unbounded component of \( M \setminus K \) has infinite volume. In particular, each end of $M$ has infinite volume.\end{teorema}

 We remark that the finiteness of the end volume is proved by Frensel for a hypersurface $M$ with mean curvature vector field bounded in norm immersed in a manifold $N$ with bounded geometry, using a different technique (see \cite[Chapter 3, Theorem 3]{frensel}). 
 
 Moreover, for $H_\varphi=0$ (strong stability), an analogous result is proved in \cite[Corollary 19]{chodoshli} for all dimensions.

To prove Theorem \ref{th:infinitevol} we need the following technical lemma, that is a direct application of Moser iteration and the Sobolev inequality \eqref{eq:sobolev}. 
 \begin{lemma}
 \label{lemma:moser}
     Assume that \(u \ge 0\) is subharmonic (i.e., \(\Delta u \ge 0\)). Then there exists a constant $C(n, \varphi)$ such that
\begin{equation}\label{eq:moser}
|| u ||_{L^ \infty B(p, \frac{1}{2} r)}\;\le\; C(n, \varphi) \, r^{-n} || u||_{L^2(B(p, r))}.
\end{equation}
 \end{lemma}
 \begin{proof}
    Since the function $u$ is subharmonic, for every test function $\eta \in C^ \infty _c (M \setminus K)$, we get
    \begin{equation}
    \label{eq:moser1}
     0 \le \int_{M} (\Delta u) \eta = - \int_{M } \langle \nabla \eta, \nabla u \rangle.  
    \end{equation}
   We replace $\eta$ with $ u^{s-1} \eta^2$, $s > 1$. Then, \eqref{eq:moser1} is
    \begin{equation}
    \label{eq:moser2}
        0 \le \int_{M} (\Delta u) u ^{s-1} \eta^2 = - \int_{M} \langle \nabla u , \nabla (u^{s-1}\eta^2) \rangle= -\int_{M} (s-1) u^{s-2} | \nabla u|^2 \eta^2 - 2 \int_{M} u ^{s-1} \eta \langle \nabla u , \nabla \eta \rangle.
\end{equation}
    Note that  
    \begin{equation}
        \label{eq:moser4}
         | \nabla (u ^{\frac{s}{2}})|^2= \frac{s^2}{4} u^{s-2} |\nabla u|^2,
    \end{equation}
    and using \eqref{eq:moser4} in \eqref{eq:moser2} we obtain
    \begin{equation}
        \label{eq:moser6}
        \int_{M} |\nabla (u^{\frac{s}{2}})|^2 \eta^2  \le C(s) \int_{M} | \nabla (u^{\frac{s}{2}})| u^{\frac{s}{2}}\eta |\nabla \eta|.
    \end{equation}
Now, applying H\"older's inequality to the term on the right hand side of \eqref{eq:moser6}, we get
\begin{equation*}
    \label{eq:moser7}
  \int_{M} | \nabla (u^{\frac{s}{2}})| u^{\frac{s}{2}}\eta |\nabla \eta| \le \left( \int_{M}  |\nabla (u^{\frac{s}{2}})|^2 \eta^2 \right)^\frac{1}{2} \left(\int_{M} |\nabla \eta|^2 u^s\right)^\frac{1}{2}.
\end{equation*}
Thus,
\begin{equation}
    \begin{split}
        \label{eq:moser8}
     \int_{M}   |\nabla (u^{\frac{s}{2}})|^2 \eta^2 \le C(s) \int_{M} |\nabla \eta|^2 u^s
    \end{split}
\end{equation}
Note that
    \begin{equation*}
        \begin{split}
            \label{eq:moser10}
            |\nabla (u^ \frac{s}{2} \eta)|^2 = |\nabla (u^ \frac{s}{2})\eta + u^ \frac{s}{2} \nabla \eta|^2 \le 2 \left(|\nabla (u^\frac{s}{2})|^2 \eta ^2 + u^s |\nabla \eta|^2\right),
        \end{split}
    \end{equation*}
    and thus
    \begin{equation}
        \label{eq:moser11}
        |\nabla (u^\frac{s}{2})|^2 \eta^2 \ge \frac{1}{2}|\nabla (u^\frac{s}{2}\eta)|^2 - u^s |\nabla \eta|^2.
    \end{equation}
    Combining \eqref{eq:moser8} and \eqref{eq:moser11} one obtains
    \begin{equation}
        \label{eq:moser12}
        \int_{M} |\nabla (u ^\frac{s}{2} \eta)|^2 \le C(s) \int_{M}|\nabla \eta|^2 u^s.
    \end{equation}
    Now, we apply Sobolev inequality \eqref{eq:sobolev} with $f= u ^\frac{s}{2} \eta$ to the left hand side of \eqref{eq:moser12}, and we get
    \begin{equation}
        \begin{split}
            \label{eq:moser12.1}
           \left( \int_{M} (u^\frac{s}{2} \eta) ^\frac{2n}{n-1} \right)^ \frac{n-1}{n}  \le C(n, \varphi, s) \int_{M} |\nabla \eta|^2u^s,
        \end{split}
    \end{equation}
    with $C(n, \varphi, s)$ a suitable constant.
We observe that the left hand side can be written as
\begin{equation*}
    \label{eq:moser13}
    \left (\int_{M} u ^{s \frac{n}{n-1}} \eta^ \frac{2n}{n-1}\right)^ \frac{n-1}{n}
\end{equation*}
and we define the following quantities
\begin{equation*}\label{eq:moser-sk}
    \begin{cases}
       s_{k+1} = \frac{n}{n-1}s_k,\\
s_0= 2.
    \end{cases}
\end{equation*}
\begin{equation}
    r_k= r- \sum_{j=0}^k 2^{-(j+2)}r,
\end{equation}
and
\begin{equation*}
\label{eq:moser-etak}
\eta_k =
\begin{cases}
1 & \text{on } B(p,r_{k+1}),\\[4pt]
0 & \text{on } (M\setminus K)\setminus B(p,r_k)
\end{cases}
\quad\text{with}\quad
|\nabla \eta_k| \le \frac{c}{r}.
\end{equation*}
Evaluating \eqref{eq:moser12.1} in $s=s_k$, $r=r_k$ and $\eta=\eta_k$, one gets
\begin{equation*}
\begin{split}
 \label{eq:moser14}   
 \left( \int_{B(p, r_{k+1})}u^{s_{k+1}}\right)^ \frac{n-1}{n} \le C(n, \varphi,s)\left(\frac{1}{r_k}\right)^2 \int_{B(p, r_k)}u ^{s_k},  \end{split}
\end{equation*}
and taking the $\frac{1}{s_k}$-th power and denoting again, by a slight abuse of notation, the constant by $C(n, \varphi, s)$, one obtains
\begin{equation*}
    \label{eq:moser15}
    \left(\int_{B(p, r_{k+1})} u ^{s_{k+1}}\right)^ \frac{1}{s_{k+1}} \le  C(n, \varphi, s) \left(\frac{1}{r_k}\right)^ \frac{2}{s_k} \left(\int_{B(p, r_k)}u^{s_k}\right)^\frac{1}{s_k}, 
\end{equation*}
that is,
\begin{equation}
    \label{eq:moser15.1}
    ||u||_{L^{s_{k+1}}(B(p, r_{k+1}))} \le C(n, \varphi, s) \left(\frac{1}{r_k}\right)^\frac{2}{s_k} ||u||_{L^{s_k}(B(p, r_k))}.
\end{equation}
Iterating \eqref{eq:moser15.1} 
\begin{equation}
    \label{eq:moser16}
     ||u||_{L^{s_{k+1}}(B(p, r_{k+1}))} \le C(n, \varphi, s) \prod_{j=0}^k \left(\frac{1}{r_j}\right)^\frac{2}{s_j} ||u||_{L^{s_0}(B(p, r_0))}.
\end{equation}
Note that
\begin{equation}
    \label{eq:moser16.1}
    s_j= s_0 \, \left(\frac{n}{n-1}\right)^j
\end{equation}
and so
\begin{equation*}
    \label{eq:moser17}
    \sum_{j=0}^\infty \frac{2}{s_j}= \frac{2}{s_0} \sum_{j=0}^\infty\left(\frac{n-1}{n}\right)^j=  \, \frac{1}{1-\frac{n-1}{n}}=  n.
\end{equation*}
Thus, taking the limit for $k \to \infty$ of \eqref{eq:moser16}, we obtain
\begin{equation*}
    \label{eq:moser18}
    ||u||_{L^\infty(B(p, \frac{r}{2}))} \le C(n, \varphi) r^{-n} ||u||_{L^2(B(p, r))}.
\end{equation*}
\end{proof}

\textit{Proof of Theorem \ref{th:infinitevol}}
 Let \( E \) be an unbounded component of \( M \setminus K \). Suppose by contradiction
that $E$ has finite volume, i.e. \( |E| < V < \infty \). Let $C(n, \varphi)$ be as in Lemma \ref{lemma:moser}. We can choose \( r \) sufficiently large so that \( C(n, \varphi)r^{2n} > V \).
Since $E$ is unbounded and $M$ is complete, there exists \( p \in E \) such that
\[
d_M(p,\partial E) > r.
\]
Then we have
\begin{equation*}
\label{eq:infvol}
V > |E| > |B_M(p,r)| \ge C(n, \varphi)r^{2n} > V,
\end{equation*}
where the third inequality is Lemma \ref{lemma:moser} applied to $u=1$. This contradiction completes the proof.
\qed\\

We now use the result of \ref{th:infinitevol} to establish the existence of a bounded harmonic function with finite Dirichlet energy.

\begin{proposizione}
\label{pr:caoshenzhu-lemma2}
Let \(M^n\) be a complete noncompact two-sided CAMC hypersurface satisfying the
hypotheses of Theorem \ref{th:sobolev}. Assume that \(M\) has at least two ends.
Then there exists a nonconstant bounded harmonic function \(u\) on \(M\) such that
\[
    \int_M |\nabla u|^2 < \infty .
\]
\end{proposizione}
\begin{proof}
Let $p$ denote the Sobolev exponent
\[
p=
\begin{cases}
\dfrac{2n}{n-1}, & \text{if } H_\varphi\neq0,\\[1.2ex]
\dfrac{2n}{n-2}, & \text{if } H_\varphi=0 \text{ and } n\ge3.
\end{cases}
\]
From Theorem \ref{th:sobolev}, there exist a compact subset \(K\subset M\) and a
constant \(C(n, \varphi)>0\) such that
\begin{equation}
\label{eq:sobolev-outside-K}
    \left(\int_M |f|^p\right)^{2/p}
    \le C(n, \varphi)\int_M |\nabla f|^2,
    \qquad
    f \in C_c^1(M\setminus K).
\end{equation}

As in Lemma 2 of \cite{caoshenzhu}, we choose an exhaustion of \(M\) by
relatively compact smooth domains \(D_i\) such that
\[
    D_i\subset \subset  D_{i+1},
    \qquad
    \bigcup_i D_i=M.
\]
Up to increasing \(i_0\), we may also assume that
\[
    K\subset \subset  D_{i_0}.
\] 
Since \(M\) has at least two ends, Theorem \ref{th:infinitevol} implies that,
after increasing \(i_0\) if necessary, we may choose two distinct components
\(E_1\) and \(E_2\) of \(M\setminus D_{i_0}\), both of infinite volume.
\\For \(i\ge i_0\), let
\[
    \Gamma_{1,i}:=\partial D_i\cap E_1.
\]
Let \(u_i\) be the minimizer of the Dirichlet energy on \(D_i\) among all
functions with boundary values
\[
    u_i=1 \quad \text{on } \Gamma_{1,i},
    \qquad
    u_i=0 \quad \text{on } \Gamma_{j,i}, j\neq 1
\]
Equivalently, \(u_i\) solves
\[
\begin{cases}
    \Delta u_i=0 & \text{in }D_i,\\
    u_i=1 & \text{on }\Gamma_{1,i},\\
    u_i=0 & \text{on }\Gamma_{0,i}.
\end{cases}
\]
By the maximum principle,
\[
    0\le u_i\le 1 \qquad \text{on }D_i.
\]

We claim that the energies of \(u_i\) are uniformly bounded. Indeed, if \(i>k\),
we extend \(u_k\) to \(D_i\) by setting it equal to \(1\) on
\((D_i\setminus D_k)\cap E_1\) and equal to \(0\) on the remaining components of
\(D_i\setminus D_k\). This extension is an admissible competitor for the
minimization problem defining \(u_i\), and its energy is the same as that of
\(u_k\). Hence
\[
    \int_{D_i} |\nabla u_i|^2
    \le
    \int_{D_k} |\nabla u_k|^2.
\]
In particular, there exists a constant \(c>0\), independent of \(i\), such that
\begin{equation}
\label{eq:uniform-energy-ui}
    \int_{D_i} |\nabla u_i|^2 \le c .
\end{equation}
Thus, a subsequence, still denoted by \(u_i\), converges smoothly on compact
subsets of \(M\) to a harmonic function \(u\) with
\[
    0\le u\le 1 \quad \text{and } \int_M |\nabla u|^2 < c.
\]
We next show that this limiting harmonic function is nonconstant.\\ Choose a smooth cut-off function \(\eta\) such that
\[
    0\le \eta\le 1,\qquad
    \eta\equiv 0 \ \text{on a neighborhood of }K,
\]
and
\[
    \eta\equiv 1 \ \text{outside a larger compact set }K',
    \qquad
    |\nabla\eta|\le C.
\]
Set
\[
    f_i:=\eta u_i(1-u_i).
\]
Since \(u_i(1-u_i)=0\) on \(\partial D_i\), $f_i$ is an admissible test function in the Sobolev inequality. Thus, \eqref{eq:sobolev-outside-K} gives 
\[
    \left(\int_{D_i} |\eta u_i(1-u_i)|^p\right)^{2/p}
    \le
    C(n, \varphi) \int_{D_i} |\nabla(\eta u_i(1-u_i))|^2 .
\]
We now estimate the right-hand side. Since
\[
    \nabla f_i
    =
    \eta(1-2u_i)\nabla u_i
    +
    u_i(1-u_i)\nabla\eta,
\]
and \(0\le u_i\le1\), we have
\[
\begin{aligned}
    |\nabla f_i|^2
    &\le
    2\eta^2(1-2u_i)^2|\nabla u_i|^2
    +
    2u_i^2(1-u_i)^2|\nabla\eta|^2  \\
    &\le
    2\eta^2|\nabla u_i|^2
    +
    \frac18|\nabla\eta|^2 .
\end{aligned}
\]
This implies, from the definition of $\eta$ and \eqref{eq:uniform-energy-ui},
\begin{equation}
\label{eq:ui-limitato}    
\left(\int_{D_i} |\eta u_i(1-u_i)|^p\right)^{2/p}
    \le C_1 
\end{equation}
for a uniform constant $C_1$.

Suppose now that \(u\equiv c\) is constant. Passing to the limit in
\eqref{eq:ui-limitato} and using that
\(\operatorname{vol}(D_i)\to\infty\) we get that necessarily
\begin{equation}    
\label{eq:c=01}
    c=0 \quad \text{or} \quad c=1.
\end{equation}
Without loss of generality, it is enough to rule out the case
\(u\equiv1\), if \(u\equiv0\), we replace \(u_i\) by \(1-u_i\) and argue in the
same way.
Assume first that \(u\equiv0\). Replacing \(u_i\) by \(1-u_i\), we reduce to
the case \(u\equiv1\). Thus it remains to rule out this possibility.

We choose a smooth function \(\psi:M\to[0,1]\) such that
\[
    \psi\equiv0 \quad \text{in a neighborhood of }K,
\]
\[
    \psi\equiv1 \quad \text{on }E_2,
    \qquad
    \psi\equiv0 \quad \text{on every other component of }M\setminus D_{i_0}.
\]
In particular, \(\nabla\psi\) has compact support contained in \(D_{i_0}\setminus K\).

Set
\[
    g_i:=u_i\psi 
\]
and apply \eqref{eq:sobolev-outside-K}. We get
\[
    \left(\int_{D_i}|u_i \psi|^p\right)^{2/p}
    \le
    C(n, \varphi)\int_M |\nabla (u_i \psi)|^2 .
\]

Moreover,
\[
\begin{aligned}
    \int_{D_i} |\nabla(u_i\psi)|^2
    &\le
    2\int_{D_i}\psi^2|\nabla u_i|^2
    +
    2\int_{D_i}u_i^2|\nabla\psi|^2  \\
    &\le
    2\int_{D_i}|\nabla u_i|^2
    +
    2\int_M|\nabla\psi|^2  \\
    &\le C_2,
\end{aligned}
\]
where \(C_2\) is independent of \(i\). 

Since \(\psi\equiv1\) on \(E_2\), this gives
\[
    \int_{E_2\cap D_i} u_i^p \le C_2 .
\]
Passing to the limit along the exhaustion yields
\[
    \operatorname{vol}(E_2)\le C_2,
\]
which contradicts the infinite volume of \(E_2\).
\end{proof}

To prove Theorem \ref{th:main-intro} it remains to rule out the existence of such harmonic functions.

\begin{teorema}
\label{th:fuli}
 Let $M$ be a weakly stable constant anisotropic mean curvature hypersurface in $\mathbb R^{n+1}$, $n \le 6$ and assume that there exists a constant $\Lambda_n < \frac{n^2}{n^2-1}$ such that \begin{equation*}
    \label{eq:pinching-intro}
    |v|^2 \le D^2 \varphi (\nu)(v,v) \le \Lambda_n|v|^2, \forall v \in \nu ^ \perp
    \end{equation*}
    Then, $M$ has only one end.
\end{teorema}
\begin{proof}
We use the following estimate due to Shiohama–Xu (see \cite[Theorem 1]{shiohamaxu}):
\begin{equation}
    \label{eq:sx}
    \operatorname{Ric}_M \ge \frac{n-1}{n} \left(  nH^2 - \frac{n(n-2)}{\sqrt{n(n-1)}}H|\mathring{A}|- |\mathring{A}|^2\right).
\end{equation}
with $|\mathring{A}|$ being the norm of the traceless second fundamental form.\\
Assume by contradiction that $M$ has 2 ends. Then, by Proposition \ref{pr:caoshenzhu-lemma2}, there exists a harmonic function $u$ with finite energy integral. From Bochner's formula,
\begin{equation}
    \label{eq:bochner}
    \Delta |\nabla u|^2= 2\left( |\nabla^2 u|^2 + \operatorname{Ric}(\nabla u, \nabla u) \right).
\end{equation}
Applying the improved Kato inequality
\begin{equation*}
    \label{eq:kato}
    \frac{n}{n-1}|\nabla |\nabla u||^2 \le|\nabla^2 u|^2,
\end{equation*}
one obtains
\begin{equation}
    \label{eq:bochner-kato}
    |\nabla u | \Delta |\nabla u| \ge \operatorname{Ric}(\nabla u, \nabla u)+ \frac{1}{n-1}|\nabla |\nabla u||^2.
\end{equation}
Multiplying \eqref{eq:bochner-kato} by $f^2$ ($f$ with compact support), using \eqref{eq:sx}, and integrating on $M$,
\begin{equation*}\begin{split}
    \label{eq:fuli1}
 0 &\le \int_M f^2 |\nabla u| \Delta |\nabla u| - \frac{1}{n-1} \int_M f^2 |\nabla |\nabla u||^2 \\ &+\int_M \left( \frac{(n-2)\sqrt{n(n-1)}}{n}|\mathring{A}||H| + \frac{n-1}{n} |\mathring{A}|^2 - (n-1)H^2\right) f^2 |\nabla u|^2. 
\end{split}\end{equation*}
Integration by parts gives
\begin{equation*}\begin{split}
    \label{eq:fuli2}
     0 \le &- 2 \int_M f \langle \nabla f, \nabla |\nabla u|\rangle |\nabla u| - \frac{n}{n-1} \int_M f^2 |\nabla |\nabla u||^2 + \frac{n-1}{n} \int_M |\mathring{A}|^2 f^2 |\nabla u|^2\\
     &+ \int_M f^2 \left( \frac{(n-2)\sqrt{n(n-1)}}{n}|\mathring{A}||H| - (n-1)H^2\right)|\nabla u|^2.
\end{split}\end{equation*}
Applying Young's inequality $(2ab \le a^2 + b^2)$ with
\begin{equation*}
    \begin{split}
        \label{eq:fuli3}
       a=  \frac{(n-2) \sqrt{n(n-1)}}{2 \sqrt{2}\sqrt{n}}|H| \qquad \text{and} \qquad b= \frac{\sqrt{2}}{\sqrt{n}}|\mathring{A}|,
    \end{split}
\end{equation*}
one obtains
\begin{equation}
    \begin{split}
        \label{eq:fuli4}
        0 & \le - 2 \int_M f \langle \nabla f,  \nabla |\nabla u| \rangle |\nabla u| - \frac{n}{n-1} \int_M f^2 |\nabla |\nabla u||^2 + \frac{n+1}{n} \int_M |\mathring{A}|^2 f^2 |\nabla u|^2\\
        &+\int_M f^2 \left( \frac{(n-2)^2(n-1)}{8}- (n-1)\right)H^2 |\nabla u|^2=\\
        &= - 2 \int_M f \langle \nabla f,  \nabla |\nabla u| \rangle |\nabla u| - \frac{n}{n-1} \int_M f^2 |\nabla |\nabla u||^2 + \frac{n+1}{n} \int_M (nH^2+|\mathring{A}|^2) f^2 |\nabla u|^2\\
        &+ \int_M f^2 \left( \frac{(n-2)^2(n-1)}{8}- 2n \right)H^2 |\nabla u|^2.
    \end{split}
\end{equation}
We now follow the choice of test function made in
\cite[Theorem 3.1]{chengcheungzhou}. For completeness, we give the explicit
construction. Let \(r(x)=d_M(p,x)\). For \(a,b,R>0\) and \(0\le t_0\le1\), define
\(f=f(t_0,a,b,R)\) by
\[
f(t_0,a,b,R)(x)=
\begin{cases}
1,
& x\in B_p(a),\\[4pt]
\dfrac{a+R-r(x)}{R},
& x\in B_p(a+R)\setminus B_p(a),\\[8pt]
t_0\dfrac{a+R-r(x)}{R},
& x\in B_p(a+2R)\setminus B_p(a+R),\\[8pt]
-t_0,
& x\in B_p(a+2R+b)\setminus B_p(a+2R),\\[4pt]
t_0\dfrac{r(x)-(a+3R+b)}{R},
& x\in B_p(a+3R+b)\setminus B_p(a+2R+b),\\[8pt]
0,
& x\in M\setminus B_p(a+3R+b).
\end{cases}
\]
As in \cite[Theorem 3.1]{chengcheungzhou}, for fixed \(a,R>0\) one can choose
\(b>0\) and \(t_0\in[0,1]\) so that
\begin{equation}
\label{eq:zero-mean-f-grad-u}
    \int_M f|\nabla u|=0.
\end{equation}
Thus \(f|\nabla u|\) is an admissible test function for the weak stability
inequality. Applying \eqref{stability} to \(f|\nabla u|\), we obtain
\begin{equation}
\label{eq:stab-fuli}
    \int_M \frac{1}{\Lambda_n} \big(nH^2+|\mathring{A}|^2\big) f^2|\nabla u|^2
    \le
    \int_M |\nabla(f|\nabla u|)|^2 .
\end{equation}
Combining \eqref{eq:fuli4} and \eqref{eq:stab-fuli} gives
\begin{equation*}
    \begin{split}
        \label{eq:fuli5}
        0 &\le  - 2 \int_M f \langle \nabla f,  \nabla |\nabla u| \rangle |\nabla u| - \frac{n}{n-1} \int_M f^2 |\nabla |\nabla u||^2 + \frac{n+1}{ n} \Lambda_n \int_M |\nabla (f |\nabla u|)|^2\\
        &+ \int_M f^2 \left( \frac{(n-2)^2(n-1)}{8}- 2n \right)H^2 |\nabla u|^2=\\
        &= \left( -2+\frac{2(n+1)}{n}\Lambda_n \right) \int_M f \langle \nabla f,  \nabla |\nabla u| \rangle |\nabla u| + \left(-\frac{n}{n-1}+ \frac{n+1}{n}\Lambda_n\right) \int_M f^2 |\nabla |\nabla u||^2\\ &+ \frac{n+1}{n }\Lambda_n \int_M |\nabla f|^2 |\nabla u|^2
        + \int_M f^2 \left( \frac{(n-2)^2(n-1)}{8}- 2n \right)H^2 |\nabla u|^2.
    \end{split}
\end{equation*}
Now, weighted Young's inequality applied to the first term yields
\begin{equation*}
    \begin{split}
        \label{eq:fuli6}
        2 \left| \int_M  f \langle \nabla f,  \nabla |\nabla u| \rangle |\nabla u| \right| \le \varepsilon \int_M f^2 |\nabla |\nabla u||^2 + \frac{1}{\varepsilon} \int_M |\nabla f|^2 |\nabla u|^2
        \end{split}
\end{equation*}
and hence
\begin{equation}
    \begin{split}
        \label{eq:fuli7}
        0 &\le \left( - \frac{n}{n-1} + \frac{n+1}{n}\Lambda_n + \varepsilon \frac{(n+1)\Lambda_n-n}{n}\right) \int_Mf^2 |\nabla |\nabla u||^2 \\
        &+ \left( \frac{n+1}{n}\Lambda_n + \frac{(n+1)\Lambda_n -n}{n \varepsilon}\right) \int_M |\nabla f|^2 |\nabla u|^2 + \int_M f^2\left( \frac{(n-2)^2(n-1)}{8}- 2n \right)H^2 |\nabla u|^2.
         \end{split}
\end{equation}
First, note that the coefficient of the last term in the inequality \eqref{eq:fuli7} is less or equal than zero if and only if $n \le 6$.
\\Moreover, define
\begin{equation*}
    \label{eq:fuli8}
    A_ \varepsilon=  \frac{n}{n-1} - \frac{n+1}{n}\Lambda_n - \varepsilon \frac{(n+1)\Lambda_n-n}{n} \qquad \text{and} \qquad B_\varepsilon= \frac{n+1}{n}\Lambda_n + \frac{(n+1)\Lambda_n -n}{n \varepsilon}
\end{equation*}
Then, \eqref{eq:fuli7} rewrites as
\begin{equation*}
    \label{eq:fuli9}
    A_\varepsilon \int_M f^2 |\nabla |\nabla u||^2 \le B_\varepsilon \int_M |\nabla f|^2 |\nabla u|^2.
\end{equation*}
As $\varepsilon \to 0$, the coefficients $A_\varepsilon$ and $B_\varepsilon$ are greater than zero when
\begin{equation*}
    \label{eq:cond-lambda}
    \Lambda_n < \frac{n^2}{n^2-1}.
\end{equation*}
Letting $R \to \infty$, using that $u$ has finite energy, we obtain $\nabla |\nabla u| \equiv 0$ on $M$, that is $|\nabla u|$ constant.
\\Finally, if $|\nabla u| \ne 0$,
\begin{equation*}
    \operatorname{vol}(E)= \int_E d\mu= \frac{1}{|\nabla u|^2}\int_E |\nabla u|^2 \,d\mu < \infty,
\end{equation*}
contradicting Theorem \ref{th:infinitevol}. Thus $|\nabla u|\equiv 0$, and hence $u$ is constant.
This contradiction proves that $M$ has only one end.
\end{proof} 
\begin{osservazione}
We note that our theorem also applies to the special case $H_\varphi = 0$, i.e., anisotropic minimal hypersurfaces. 
As already mentioned in the Introduction, in \cite{lixia}, an analogous result was established for stable anisotropic minimal hypersurfaces $(n \le 5)$ satisfying the pinching condition \eqref{eq:pinching-intro-lambda}, with $\Lambda_4 = \frac{23}{20}$ and $\Lambda_5 = \frac{1001}{1000}$. 
Our result extends their theorem to dimension 6 and also provides a slightly improved pinching constant $ \left(\Lambda_5< \frac{25}{24}\right)$ in $\mathbb R^6$.

\end{osservazione}

\end{document}